\documentclass{amsart}

\usepackage{amsmath, enumerate, amsfonts, amssymb, amsthm, graphics, graphicx, verbatim, caption, subcaption, MnSymbol, booktabs, multirow, array, hyperref, float, comment}
\usepackage[numbers]{natbib}
\usepackage[table]{xcolor}
\usepackage[english]{babel}
\usepackage[margin=0.95in]{geometry}
\usepackage{tikz, float}
\usetikzlibrary{shapes}
\usetikzlibrary{positioning}
\usetikzlibrary{decorations.markings}
\usetikzlibrary{arrows}
\tikzstyle{subgroup}=[scale=1]

\newtheorem{prevtheorem}{Theorem}

\newtheorem{proposition}{Proposition}
\newtheorem{corollary}{Corollary}
\newtheorem{lemma}{Lemma}

\def\CD#1{\mathcal{CD}#1}

\title{\bf Central Products and the Chermak--Delgado Lattice}

\author{William Cocke}
\address{School of Computer and Cyber Sciences, Augusta University, Augusta, GA 30901}
\email{wcocke@augusta.edu}

\author{Ryan McCulloch}
\address{Division of Mathematics \& Natural Sciences, Elmira College, Elmira, NY 14901}
\email{rmcculloch@elmira.edu}

\date{\today}

\begin{document}

\begin{abstract}
The Chermak--Delgado lattice of a finite group is a modular, self-dual sublattice of the lattice of subgroups. We prove that the Chermak--Delgado lattice of a central product contains the product of the Chermak--Delgado lattices of the relevant central factors.  Furthermore, we obtain information about heights of elements in the Chermak--Delgado lattice relative to their heights in the Chermak--Delgado lattices of central factors.  We also explore how the central product can be used as a tool in investigating Chermak--Delgado lattices.  
\end{abstract}



\maketitle

\section{Introduction.}
The Chermak--Delgado lattice consists of subgroups of a finite group that have maximal Chermak--Delgado measure. Due to the many unique properties of the Chermak--Delgado lattice, it has attracted attention from researchers interested in lattice theory, general finite group theory, and centralizers of groups. 

Originally defined by Chermak and Delgado \cite{Che_89}, the Chermak--Delgado lattice for a finite group $G$ is defined using the so-called Chermak--Delgado measure $m$ which takes subgroups of $G$ to positive integers via the formula \[m_G(H) = |H| \cdot |\textbf{C}_G(H)|.\] It is quite interesting, and perhaps counter-intuitive that the subgroups with maximal value of $m$ in a group form a sublattice of the subgroup lattice of $G$. Recall that the subgroup lattice of $G$ is the poset of subgroups of $G$ with the operations of meet and join defined by subgroup intersection and subgroup generated by respectively, i.e., $H \wedge K = H \cap K$ and $H\vee K = \langle H, K\rangle$. Surprisingly, when $H$ and $K$ are in the Chermak--Delgado lattice, we have that $\langle H , K \rangle = HK$, i.e., the set theoretic product is actually a subgroup. Furthermore if $H$ is in the Chermak--Delgado lattice then $\textbf{C}_G(H)$ is in the Chermak--Delgado lattice. In addition, we know that for all $H$ in the Chermak--Delgado lattice of $G$, the subgroup $H$ must contain $\textbf{Z}(G)$, the center of $G$. The many properties we presented in this paragraph can be found in Isaacs \cite[Section 1.G]{FGT}. Because of these properties as well as current research about the Chermak--Delgado lattice for various groups, questions about the Chermak--Delgado lattice make for good research projects that can be accessible to students \cite{Wil_16, Bru_16}. 

We write $m^*(G)$ for the maximum value that $m$ takes on a finite group $G$, and $\CD(G)$ for the Chermak--Delgado lattice of $G$. Hence for a subgroup $H$ of $G$ we have that $H\in \CD(G)$ if and only if $m_G(H) = m^*(G)$. 

Recent work on the Chermak--Delgado lattice can be broadly classified as coming in two broad themes. While not exhaustive, we list some references below. \begin{enumerate} \item Showing that certain types of structures occur as Chermak--Delgado lattices for various group: chains \cite{Bre_14}, antichains \cite{Bre_14_2}, diamonds \cite{An_15}, all subgroups containing $\textbf{Z}(G)$ in a group \cite{Tar_18}; or how the Chermak--Delgado lattice relates to certain families of groups \cite{Gla_06}, \cite{Zuc_18}, \cite{An_22_2}, \cite{Mcc_18}.
\item Showing properties of the Chermak--Delgado lattice in general: the Chermak--Delgado lattice of a direct product is the direct product of the Chermak--Delgado lattices of the factors \cite{Bre_12}. All subgroups of the Chermak--Delgado lattice of a finite group are subnormal in that group with subnormal depth bounded by their relevant position within the lattice \cite{Bre_12, Coc_20}. Most recently, the question was asked about how many groups are not contained in the Chermak--Delgado lattice \cite{Fas_22}. 
\end{enumerate}

Related to the direct product is the central product. Recall that a group $G$ is a central product of two of its subgroups $A$ and $B$ if $G=AB$ and $[A,B]=1,$ i.e., $ab=ba$ for all $a\in A$ and $b\in B$.  

We prove that the Chermak--Delgado lattice of a central product contains the product of the Chermak--Delgado lattices of the relevant central factors. 
\begin{prevtheorem}\label{thm:central_product}
Let $G$ be a finite group and $A,B \leq G$ such that $G$ is a central product of $A$ and $B$. Then $\CD(A)\cdot \CD(B) \subseteq \CD(G)$. Furthermore, the top (resp.\ bottom) of $\CD(G)$ is equal to the product of the tops (resp.\ bottoms) of $\CD(A)$ and $\CD(B)$. 
\end{prevtheorem} 

Hence we have the following corollary. 

\begin{corollary}
If $G$ is a central product of $A$ and $B$ then $m^*(G) = \dfrac{m^*(A) \cdot m^*(B)}{|A\cap B|^2}.$
\end{corollary} 

In addition, we have the following structural information about the Chermak--Delgado lattices of central products regarding the heights and depths of elements in the product. 

Recall that the \emph{depth of} $H\in \CD(G)$ is the length of a maximal chain 
\[H= G_0 < G_1 < \dots < G_n = T_G\] of elements of $\CD(G)$ where $T_G$ is the top element; the \emph{height of} $H\in \CD(G)$ is the length of a maximal chain 
\[B_G = G_0 < G_1 < \dots < G_n = H\] of elements of $\CD(G)$ where $B_G$ is the bottom element.  The \emph{height of} $\CD(G)$ is the height of the top element (or, equivalently, the depth of the bottom element).  These quantities are well-defined since $\CD(G)$ is a modular lattice, and so all of the maximal chains between two fixed elements are of the same length.

\begin{prevtheorem} \label{thm: levels}
If $G$ is a central product of $A$ and $B$, then the height of the Chermak--Delgado lattice of $G$ is equal to the sum of the heights of the Cheramk--Delgado lattices of $A$ and $B$ respectively. Moreover, an element $HK\in \CD(G)$ with $A \cap B \leq H \leq A$ and $A \cap B \leq K \leq B$ has height (resp.\ depth) equal to the sum of the heights (resp.\ depths) of $H\in \CD(A)$ and $K\in \CD(B)$. 
\end{prevtheorem}

In Section \ref{sec: corollaries} we show how the central product can be used to prove results about the Chermak--Delgado lattice. 

\section{Central Products and the Chermak--Delgado Lattice.}\label{sec: main}

A group $G$ is a central product of $A$ and $B$, which are subgroups of $G$, if $G=AB$ and $[A,B]=1$.  Note for every group $G$ we have the central product $G=G \textbf{Z}(G)$.  Also note that if $G=AB$ is a central product, then both $A$ and $B$ are normal subgroups of $G$ and that $A \cap B \subseteq \textbf{Z}(G)$.  Also for any $X_1 \leq A$ and for any $X_2 \leq B$, we have that $\textbf{C}_G(X_1X_2) = \textbf{C}_A(X_1)\textbf{C}_B(X_2)$. This last observation is almost enough to prove Theorem \ref{thm:central_product}.

\begin{proposition} \label{prop: almost}
Suppose a finite group $G$ is a central product of $A$ and $B$. If $HK \in \mathcal{C}\mathcal{D}(G)$ for some $H \leq A$ and some $K \leq B$, then $m^*(G) = \dfrac{m^*(A) \cdot m^*(B)}{|A \cap B|^2}$, $\mathcal{C}\mathcal{D}(A) \cdot \mathcal{C}\mathcal{D}(B) \subseteq \mathcal{C}\mathcal{D}(G)$, and $H(A \cap B) \in \mathcal{C}\mathcal{D}(A)$ and $K(A \cap B) \in \mathcal{C}\mathcal{D}(B)$.
\end{proposition}

\begin{proof}
Suppose $HK \in \mathcal{C}\mathcal{D}(G)$.  Then $A \cap B \leq \textbf{Z}(G) \leq HK$.  And so $HK = H(A \cap B)K(A \cap B)$.  Also, $A \cap B \leq \textbf{C}_A(H)$ and $A \cap B \leq \textbf{C}_B(K)$.

Then $$m^*(G) = m_G(HK) = |HK| \cdot |\textbf{C}_G(HK)| = |H(A \cap B)K(A \cap B)| \cdot |\textbf{C}_A(H)\textbf{C}_B(K)| =$$

$$\frac{|H(A \cap B)| \cdot |K(A \cap B)|}{|A \cap B|} \cdot \frac{|\textbf{C}_A(H)| \cdot |\textbf{C}_B(K)|}{|A \cap B|} = \frac{m_A(H(A \cap B)) \cdot m_B(K(A \cap B))}{|A \cap B|^2}.$$ 

Let $X \in \mathcal{C}\mathcal{D}(A)$ and $Y \in \mathcal{C}\mathcal{D}(B)$.  Then $A \cap B \leq \textbf{Z}(A) \leq X$, $A \cap B \leq \textbf{Z}(B) \leq Y$, $A \cap B \leq \textbf{C}_A(X)$, and $A \cap B \leq \textbf{C}_B(Y)$. And so

$$m_G(XY) = |XY| \cdot |\textbf{C}_G(XY)| = |XY| \cdot |\textbf{C}_A(X)\textbf{C}_B(Y)| = \frac{|X| \cdot |Y|}{|A \cap B|} \cdot \frac{|\textbf{C}_A(X)| \cdot |\textbf{C}_B(Y)|}{|A \cap B|} = $$

$$\frac{m_A(X) \cdot m_B(Y)}{|A \cap B|^2} = \dfrac{m^*(A) \cdot m^*(B)}{|A \cap B|^2}  \geq \frac{m_A(H(A \cap B)) \cdot m_B(K(A \cap B))}{|A \cap B|^2} = m^*(G).$$ 

Thus $m_G(XY) = \dfrac{m^*(A)m^*(B)}{|A \cap B|^2} = m^*(G)$, and we have that $\mathcal{C}\mathcal{D}(A) \cdot \mathcal{C}\mathcal{D}(B) \subseteq \mathcal{C}\mathcal{D}(G)$, and also $H(A \cap B) \in \mathcal{C}\mathcal{D}(A)$ and $K(A \cap B) \in \mathcal{C}\mathcal{D}(B)$.
\end{proof}

As it stands, Proposition \ref{prop: almost} proves Theorem \ref{thm:central_product} modulo the existence of a subgroup of the form $HK\in \CD(G)$. While Theorem \ref{thm:central_product} tells us what such a subgroup could be, i.e., take the product of the tops of $\CD(A)$ and $\CD(B)$, it is difficult to directly show that this product is in $\CD(G)$. Instead, we utilize some lemmata comparing information about how the Chermak--Delgado measure behaves across various subgroups. 

The following lemma is found in Issacs \cite[Lemma 1.43]{FGT}.

\begin{lemma} \label{lem: iss}
Suppose that $G$ is a finite group. If $H,K \leq G$, then 

$$m_G(H)\cdot m_G(K) \leq m_G(\langle H,K \rangle)\cdot m_G(H \cap K).$$

Moreover, equality holds if and only if $\langle H,K \rangle = HK$ and $\textbf{C}_G(H \cap K) = \textbf{C}_G(H)\textbf{C}_G(K)$.

\end{lemma}
\begin{proof}
\begin{flalign*}
m_G(H) \cdot m_G(K) &= |H| \cdot |\textbf{C}_G(H)| \cdot |K| \cdot |\textbf{C}_G(K)| \\
& = |H| \cdot |K| \cdot |\textbf{C}_G(H)| \cdot |\textbf{C}_G(K)|\\
& = |HK| \cdot |H \cap K | \cdot |\textbf{C}_G(H) \textbf{C}_G(K)| \cdot |\textbf{C}_G(H) \cap \textbf{C}_G(K)|\\
& \leq |\langle H, K \rangle |\cdot |\textbf{C}_G(H)\cap \textbf{C}_G(K)| \cdot  |H\cap K| \cdot |\textbf{C}_G(H) \textbf{C}_G(K)|\\
& \leq |\langle H, K \rangle |\cdot |\textbf{C}_G(H)\cap \textbf{C}_G(K)| \cdot  |H\cap K| \cdot |\textbf{C}_G(H\cap K)|\\
& = m_G(\langle H,K\rangle) \cdot m_G(H\cap K). 
\end{flalign*}
\end{proof}

The following lemma generalizes a result of An \cite[Lemma 3.2]{An_22}.

\begin{lemma} \label{lem: an1}
Suppose that $G$ is a finite group.  If $K \leq X \leq H \leq G$, then

$$\frac{m_H(K)}{m_G(K)} \leq \frac{m_H(X)}{m_G(X)}.$$

Moreover, equality holds if and only if $\textbf{C}_G(K) \subseteq H\textbf{C}_G(X)$.
\end{lemma}

\begin{proof}
Note that

$$\frac{m_H(K)}{m_G(K)} = \frac{|K|\cdot |\textbf{C}_H(K)|}{|K|\cdot |\textbf{C}_G(K)|} = \frac{|\textbf{C}_H(K)|}{|\textbf{C}_G(K)|} = \frac{|H \cap \textbf{C}_G(K)|}{|\textbf{C}_G(K)|} = \frac{|H|}{|H\textbf{C}_G(K)|}$$

and similarly

$$\frac{m_H(X)}{m_G(X)} = \frac{|X|\cdot |\textbf{C}_H(X)|}{|X|\cdot |\textbf{C}_G(X)|} = \frac{|\textbf{C}_H(X)|}{|\textbf{C}_G(X)|} = \frac{|H \cap \textbf{C}_G(X)|}{|\textbf{C}_G(X)|} = \frac{|H|}{|H\textbf{C}_G(X)|}.$$

Since $K \leq X$, $\textbf{C}_G(X) \leq \textbf{C}_G(K)$.  So $H\textbf{C}_G(X) \subseteq H\textbf{C}_G(K)$ and we have $|H\textbf{C}_G(X)| \leq |H\textbf{C}_G(K)|$, where equality holds if and only if $\textbf{C}_G(K) \subseteq H\textbf{C}_G(X)$.  Thus

$$\frac{m_H(K)}{m_G(K)} = \frac{m_H(X) \cdot |H\textbf{C}_G(X)|}{m_G(X) \cdot |H\textbf{C}_G(K)|} \leq \frac{m_H(X)}{m_G(X)},$$

where equality holds if and only if $\textbf{C}_G(K) \subseteq H\textbf{C}_G(X)$.
\end{proof}

The last lemma we will use to prove Theorem \ref{thm:central_product} is one step below that of a central product. Here the group $G$ will be equal to $H\textbf{C}_G(X)$ where $X\leq H \leq G$.  Note that a group $G$ is a central product if $G = H\textbf{C}_G(H)$ for some subgroup $H$ of $G$.  This lemma generalizes a result of An \cite[Lemma 3.3]{An_22}.

\begin{lemma} \label{lem: an2}
Suppose that $G$ is a finite group and $X \leq H \leq G$ such that $G = H\textbf{C}_G(X)$.  If $X \in \mathcal{C}\mathcal{D}(H)$, then for every $Y \in \mathcal{C}\mathcal{D}(G)$, we have $\langle X,Y \rangle \in \mathcal{C}\mathcal{D}(G)$ and $X \cap Y \in \mathcal{C}\mathcal{D}(H)$.  Furthermore, $\langle X,Y \rangle = XY$ and $\textbf{C}_G(X \cap Y) = \textbf{C}_G(X)\textbf{C}_G(Y)$.
\end{lemma}

\begin{proof}
Let $Y \in \mathcal{C}\mathcal{D}(G)$.  Since $\textbf{C}_G(X \cap Y) \leq G = H\textbf{C}_G(X)$, by Lemma \ref{lem: an1},

$$\frac{m_H(X \cap Y)}{m_G(X \cap Y)} = \frac{m_H(X)}{m_G(X)}.$$

Since $X \in \mathcal{C}\mathcal{D}(H)$, $m_H(X \cap Y) \leq m_H(X)$.  It follows that $m_G(X \cap Y) \leq m_G(X)$.  By Lemma \ref{lem: iss},

$$m_G(X)\cdot m_G(Y) \leq m_G(\langle X,Y \rangle)\cdot m_G(X \cap Y).$$

It follows that $m_G(\langle X,Y \rangle) \geq m_G(Y) = m^*(G)$.  Hence $m_G(\langle X,Y \rangle) = m_G(Y)$ and $\langle X,Y \rangle \in \mathcal{C}\mathcal{D}(G)$. Hence  $m_G(X) \leq m_G(X \cap Y)$, and so $m_G(X) = m_G(X \cap Y)$.

And so

$$m_G(X)\cdot m_G(Y) = m_G(\langle X,Y \rangle)\cdot m_G(X \cap Y),$$

and by Lemma \ref{lem: iss} we have that $\langle X,Y \rangle = XY$ and $\textbf{C}_G(X \cap Y) = \textbf{C}_G(X)\textbf{C}_G(Y)$.

Finally, since 

$$\frac{m_H(X \cap Y)}{m_G(X \cap Y)} = \frac{m_H(X)}{m_G(X)},$$

we conclude that $m_H(X \cap Y) = m_H(X) = m^*(H)$.  Hence $X \cap Y \in \mathcal{C}\mathcal{D}(H)$.
\end{proof}

We can now prove Theorem \ref{thm:central_product} which states that for a finite group $G=AB$ a central product, we have $\CD(A) \cdot \CD(G) \subseteq \CD(G)$, and furthermore we have that the top (resp.\ bottom) of $\CD(G)$ is equal to the product of the tops (resp.\ bottoms) of $\CD(A)$ and $\CD(B)$. 

\begin{proof}[Proof of Theorem \ref{thm:central_product}]
We write $T_A$, $T_B$, $T_G$ for the top elements of the Chermak--Delgado lattices of $A$, $B$ and $G$ respectively. Similarly $B_A, B_B, B_G$ refer to the bottom elements of these lattices. 

For any $X_1 \leq A$, $G=A\textbf{C}_G(X_1)$, and for any $X_2 \leq B$, $G=B\textbf{C}_G(X_2)$, and so Lemma \ref{lem: an2} applies for any $X_1 \in \mathcal{C}\mathcal{D}(A)$ and any $X_2 \in \mathcal{C}\mathcal{D}(B)$, with any $Y \in \mathcal{C}\mathcal{D}(G)$.

By Lemma \ref{lem: an2}, $T_AT_G \in \mathcal{C}\mathcal{D}(G)$, and so $T_A \leq T_G$.  Similarly $T_B \leq T_G$.  So $T_AT_B \leq T_G$.

By Lemma \ref{lem: an2}, $B_A \cap B_G \in \mathcal{C}\mathcal{D}(A)$, and so $B_A = B_A \cap B_G$.  By Lemma \ref{lem: an2}, $\textbf{C}_G(B_A) = \textbf{C}_G(B_A \cap B_G) = \textbf{C}_G(B_A)\textbf{C}_G(B_G) = \textbf{C}_G(B_A)T_G$.  So $T_G \leq \textbf{C}_G(B_A) = T_AB$.  Similarly $T_G \leq AT_B$.  We see that $T_G \leq AT_B \cap T_AB = T_AT_B(A \cap B)$.  And $A \cap B \leq \textbf{Z}(A) \leq T_A$, and so $T_G \leq T_AT_B$.

Thus, $T_AT_B = T_G$, and so $B_G = \textbf{C}_G(T_G) = \textbf{C}_A(T_A)\textbf{C}_B(T_B) = B_AB_B$.  We apply Proposition \ref{prop: almost} to complete the proof. 
\end{proof}

We provide two examples of equality in Theorem \ref{thm:central_product}.

\begin{proposition} \label{prop: direct}
Suppose that a finite group $G = A \times B$ is a direct product.  Then $\CD(G) = \CD(A) \cdot \CD(B)$.
\end{proposition}

\begin{proposition} \label{prop: nonproper}
Suppose that a finite group $G = AB$ with $B \leq \textbf{Z}(G)$.  Then $\CD(G) = \CD(A) \cdot \{ B \}$.
\end{proposition}

Proposition \ref{prop: direct} appears in \cite{Bre_12}.  We shall prove Proposition \ref{prop: nonproper} in a moment, but first a few facts regarding central products and group products in general.

\begin{lemma} \label{lem: prod}
Suppose that a group $G=AB$.  Suppose $A \cap B \leq H_1 \leq A$, $A \cap B \leq H_2 \leq A$, $A \cap B \leq K_1 \leq B$, and $A \cap B \leq K_2 \leq B$.  If $H_1K_1 \subseteq H_2K_2$, then $H_1 \leq H_2$ and $K_1 \leq K_2$; and if equality holds then $H_1 = H_2$ and $K_1 = K_2$.
\end{lemma}

\begin{proof}
Suppose $H_1K_1 \subseteq H_2K_2$, and let $h_1 \in H_1$ and $k_1 \in K_1$ be arbitrary.  Then $h_1k_1 = h_2k_2$ for some $h_2 \in H_2$ and some $k_2 \in K_2$.  And so $x = h_2^{-1}h_1 = k_2k_1^{-1} \in H_1 \cap K_1 \leq A \cap B$.  And since $A \cap B \leq H_2$ and $A \cap B \leq K_2$, we have that $h_1 = h_2x \in H_2$ and $k_1 = x^{-1}k_2 \in K_2$.  Hence $H_1 \leq H_2$ and $K_1 \leq K_2$.

If $H_1K_1$ and $H_2K_2$ coincide, then $H_1K_1 \subseteq H_2K_2$ implies that $H_1 \leq H_2$ and $K_1 \leq K_2$, and similarly $H_2K_2 \subseteq H_1K_1$ implies that $H_2 \leq H_1$ and $K_2 \leq K_1$.
\end{proof}

Note that Lemma \ref{lem: prod} applies for any finite central product $G=AB$ with $H_1,H_2 \in \CD(A)$ and $K_1,K_2 \in \CD(B)$, since $A \cap B \leq Z(A)$ and $A \cap B \leq Z(B)$ and all subgroups in the Chermak--Delgado lattice contain the center.  It follows that for any finite central product $G=AB$, we have that the sum of the heights of $\CD(A)$ and $\CD(B)$ cannot exceed the height of $\CD(G)$.  For if $CD(A)$ has height $n$, and say $B_A = G_0 < G_1 < \cdots < G_n = T_A$ is a chain in $\CD(A)$, and if $CD(B)$ has height $m$, and say $B_B = H_0 < H_1 < \cdots < H_m = T_B$ is a chain in $\CD(B)$, then we have by Lemma \ref{lem: prod} that $G_0H_0 < G_1H_0 < \cdots < G_nH_0 < G_nH_1 < \cdots < G_nH_m$, and by Theorem \ref{thm:central_product}, this chain of length $n+m$ lives in $\CD(G)$. Of course, our Theorem \ref{thm: levels} gives the precise relationship between heights in $\CD(A)$, $\CD(B)$, and $\CD(G)$.

If a group $G=AB$ is a central product, then for any $U\leq G$, we define $\pi_A(U) = \{ a \, | \, g=ab \text{ with } g \in U \text{ and } a \in A \text{ and } b \in B \}$ and $\pi_B(U) = \{ b \, | \, g=ab \text{ with } g \in U \text{ and } a \in A \text{ and } b \in B \}$.

\begin{lemma} \label{lem: pi_is_subgroup}
Suppose a group $G=AB$ is a central product, and $U \leq G$.  Then $\pi_{A}(U)$ is a subgroup of $A$ and, furthermore, $A \cap B \leq \pi_{A}(U) \leq A$.  A similar result is true for $\pi_{B}(U)$.
\end{lemma}

\begin{proof}
We prove the assertion for $\pi_{A}(U)$, and the assertion for $\pi_{B}(U)$ is similar.

Clearly $\pi_{A}(U) \subseteq A$.  We show that $\pi_{A}(U)$ is a subgroup of $A$.  Note that $\pi_{A}(U)$ is nonempty as $1 \in \pi_{A}(U)$.  Let $x,y \in \pi_{A}(U)$.  We show that $xy^{-1} \in \pi_{A}(U)$.

So there exists $u \in U$ and $v \in U$ so that $u=xb_1$ and $v=yb_2$ with $b_1,b_2 \in B$.

And so $v^{-1} = y^{-1}{b_2}^{-1}$, and we have $uv^{-1} = xy^{-1} b_1{b_2}^{-1}$ with $xy^{-1} \in A$ and $b_1{b_2}^{-1} \in B$.  And since $uv^{-1} \in U$, it follows that $xy^{-1} \in \pi_{A}(U)$.

Finally, we show that $A \cap B \leq \pi_{A}(U)$.  Let $z \in A \cap B$. Then $z \in A$ and $z=b$ for some $b \in B$.  And so $z^{-1} = b^{-1}$.  Let $h \in \pi_{A}(U)$.  Then there is $u \in U$ so that $u=hb'$ for some $b' \in B$.  And so $\displaystyle u = uzz^{-1} = hb'zb^{-1}= hz b'b^{-1}$ and we have $hz \in A$ and $b'b^{-1} \in B$.  Hence $hz \in \pi_{A}(U)$.  And since $\pi_{A}(U)$ is a subgroup, $h^{-1}hz = z \in \pi_{A}(U)$.
\end{proof}

\begin{lemma} \label{lem: pi}
Suppose a group $G=AB$ is a central product.  Suppose $K \leq B$ and suppose $A \cap B \leq U \leq G$.  If $K \leq U \leq AK$, then $U = \pi_A(U)K$.
\end{lemma}

\begin{proof}
If $K \leq U \leq AK$, then, of course, $K \leq U \leq \pi_A(U)K$.  And so to prove that $U = \pi_A(U)K$, it suffices to prove that $\pi_A(U) \leq U$.

Let $ab \in U$ with $a \in A$ and $b \in B$.  Since $U \leq AK$, $ab=a'k$ for some $a' \in A$ and some $k \in K$.  So $a' =az$ for some $z \in A \cap B$.  Since $K \leq U$, $a'k{k}^{-1} = a'=az \in U$.  Since $A \cap B \leq U$, $azz^{-1} = a \in U$.  Hence $\pi_A(U) \leq U$.
\end{proof}

We now prove Proposition \ref{prop: nonproper} which states that for a finite group $G=AB$ with $B \leq \textbf{Z}(G)$, we have that $\CD(G) = \CD(A) \cdot \{ B \}$.

\begin{proof}[Proof of Proposition \ref{prop: nonproper}]
We have that $G=AB$ is a central product, and since $B$ is abelian, $\CD(B) = \{ B \}$.   It follows from Theorem \ref{thm:central_product} that $\CD(A) \cdot \{ B \} \subseteq \CD(G)$.

Let $U \in \CD(G)$.  Then $\textbf{Z}(G) \leq U$, and so $B \leq U$, and so by Lemma \ref{lem: pi}, $U = \pi_A(U)B$.  By Proposition \ref{prop: almost}, we have that $\pi_A(U)(A \cap B) = \pi_A(U) \in \CD(A)$ and $B \in \CD(B)=\{B\}$, and so $U = \pi_A(U)B \in \CD(A)\cdot\{B\}$.
\end{proof}

\begin{corollary} \label{cor: G_in_CD}
Suppose a finite group $G=AB$ is a a central product.  Then $G \in \CD(G)$ if and only if $A \in \CD(A)$ and $B \in \CD(B)$.
\end{corollary}

\begin{proof}
If $A \in \CD(A)$ and $B \in \CD(B)$, then by Theorem \ref{thm:central_product}, $AB=G \in \CD(G)$.

Conversely, suppose $G \in \CD(G)$.  Then by Theorem \ref{thm:central_product}, we have $T_AT_B = G \in \CD(G)$, where $T_A$ is the top element of $\CD(A)$ and $T_B$ is the top element of $\CD(B)$.  Now $A \cap B \leq Z(A) \leq T_A$, and $A \cap B \leq Z(B) \leq T_B$.  Since $T_AT_B = AB$, it follows from Lemma \ref{lem: prod} that $T_A = A$ and $T_B = B$.
\end{proof}

We say that a central product $G=AB$ is proper if $\textbf{Z}(G) < A < G$ and $\textbf{Z}(G) < B < G$, and in such case we say that the group $G$ admits a proper central product.

Note that if a group $G=AB$ is a central product with one of $A$ or $B$ abelian, say $B$, then $B \leq \textbf{Z}(G)$, and we are in the situation of Proposition \ref{prop: nonproper}.  

Given a finite group $G$ which admits a proper central product, one wonders if $\CD(G)$ is equal to the subgroup collection of $G$ that is generated by all $\CD(X) \cdot \CD(Y)$ where $G=XY$ a proper central product.  The answer is ``no'' in general.

\begin{proposition} \label{prop: antichain}
There exists a finite group $G$ with $G \in \CD(G)$ so that $\CD(G)$ has height $2$ and $\CD(G)$ possesses both abelian and nonabelian subgroups of height $1$.
\end{proposition}

See \cite[Proposition 10]{Bre_14_2} for an explicit construction.

\begin{proposition} \label{prop: bad_antichain}
There exists a finite group $G$ that admits a proper central product, so that for $\mathcal{C}$ the subgroup collection of $G$ generated by all $\CD(X) \cdot \CD(Y)$ where $G=XY$ is a proper central product, we have that $\mathcal{C}$ is not equal to $\CD(G)$.
\end{proposition}

\begin{proof}
Let $G$ be a group as in Proposition \ref{prop: antichain}.  Then there exists $H \in \CD(G)$ nonabelian of height $1$, and so $\textbf{C}_G(H) \in \CD(G)$ is nonabelian of height $1$, and so the central product $G = H \textbf{C}_G(H)$ is proper.

If $G=AB$ is an arbitrary proper central product, then by Corollary \ref{cor: G_in_CD}, we have $A \in \CD(A)$ and $B \in \CD(B)$.  By Theorem \ref{thm:central_product} we have that $\CD(A) \cdot \CD(B) \subseteq \CD(G)$, and since $\CD(G)$ has height $2$, it follows that $\CD(A) = \{A, \textbf{Z}(A)\}$ and $\CD(B) = \{B,\textbf{Z}(B)\}$ both have height $1$.  Thus the only abelian subgroup in $\CD(A) \cdot \CD(B)$ is $\textbf{Z}(A)\textbf{Z}(B)=\textbf{Z}(G)$.  And so if $\mathcal{C}$ is the subgroup collection of $G$ generated by all $\CD(X) \cdot \CD(Y)$ where $G=XY$ is a proper central product, then the only abelian subgroup in $\mathcal{C}$ is $\textbf{Z}(G)$.  Thus $\mathcal{C}$ is not equal to $\CD(G)$.
\end{proof}

We continue with an application of Theorem \ref{thm:central_product}.

\begin{proposition} \label{prop: small_CD_lattice}
Suppose that $G$ is a finite group and suppose that $\CD(G) = \{\textbf{Z}(G), G \}$.  Then $G$ admits no proper central product.
\end{proposition}

\begin{proof}
Note that if $G=\textbf{Z}(G)$, then $G$ is abelian and the result is true.  Suppose $G=AB$ is an arbitrary central product.  We will show that one of $A$ or $B$ is abelian, and thus $G$ admits no proper central product.

By Corollary \ref{cor: G_in_CD}, we have $A \in \CD(A)$ and $B \in \CD(B)$.  By Theorem \ref{thm:central_product} we have that $\CD(A) \cdot \CD(B) \subseteq \CD(G)$, and since $\CD(G)$ has height $1$, it follows that one of $\CD(A)$ or $\CD(B)$ has height $0$, say $\CD(A)$.  So $\CD(A) = \{ A \}$, and thus $A$ is abelian.
\end{proof}

We say that a subgroup $A \in \CD(G)$ is an atom if the height of $A \in \CD(G)$ is $1$, and we say that a subgroup $B \in \CD(G)$ is a coatom if the depth of $B \in \CD(G)$ is $1$.

\begin{lemma}
Suppose that $G$ is a finite group with $G\in \CD(G)$ and suppose that $\CD(G)$ has height greater than $1$. If $A$ is a nonabelian atom in $\CD(G)$, then $A \textbf{C}_G(A)$ is a proper central product. 
\end{lemma}

\begin{proof}
We note that $\textbf{C}_G(A)$ is a coatom of $\CD(G)$. Since $A$ is nonabelian it is not a subgroup of $\textbf{C}_G(A)$. Since $\textbf{C}_G(A)$ is a coatom, the only element of $\CD(G)$ properly containing it is $G$. Thus $A \textbf{C}_G(A)$ is all of $G$. Moreover, $\textbf{C}_G(A)$ is not abelian, else it would be contained in $C_G(C_G(A)) = A$, which would imply it is $A$ since $A$ is an atom, but this implies that $A = \textbf{C}_G(A)$ which implies that $A$ is abelian. So the central product $G= A \textbf{C}_G(A)$ is proper. 
\end{proof}

\begin{lemma} \label{lem: abel_atoms}
Suppose that $G$ is a finite group with $G \in \CD(G)$ and suppose that $\CD(G)$ has height greater than $1$.  If $G$ admits no proper central product, then every atom in $\CD(G)$ is abelian.
\end{lemma}

\begin{proof}
If some atom $A \in \CD(G)$ is nonabelian, then $G=A\textbf{C}_G(A)$ is a proper central product.
\end{proof}

\begin{lemma}
Suppose that $G$ is a finite group with $G \in \CD(G)$ and suppose that $\CD(G)$ has height greater than $1$.  If $C$ is a coatom in $\CD(G)$, then either $\textbf{C}_G(C) = \textbf{Z}(C)$, or $C \textbf{C}_G(C)$ is a proper central product. 
\end{lemma}

We now prove Theorem \ref{thm: levels} which states that for a finite group $G=AB$ a central product, the height of the Chermak--Delgado lattice of $G$ is equal to the sum of the heights of the Cheramk--Delgado lattices of $A$ and $B$ respectively.  Moreover, an element $HK \in \CD(G)$ with $A \cap B \leq H\leq A$ and $A \cap B \leq K\leq B$ has height (resp.\ depth) equal to the sum of the heights (resp.\ depths) of $H\in \CD(A)$ and $K\in\CD(B)$.

\begin{proof}[Proof of Theorem \ref{thm: levels}]
We first prove the assertion in regards to height.  It follows from Proposition \ref{prop: almost} that given $HK \in \CD(G)$ with $A \cap B \leq H\leq A$ and $A \cap B \leq K\leq B$, we have $H\in \CD(A)$ and $K\in\CD(B)$, and it follows from Theorem \ref{thm:central_product} that such $HK \in \CD(G)$ exists, namely the product of the top element of $\CD(A)$ and the top element of $\CD(B)$.

Suppose $HK \in \CD(G)$ with $A \cap B \leq H\leq A$ and $A \cap B \leq K\leq B$, and suppose $H\in \CD(A)$ has height $i$ and $K\in\CD(B)$ has height $j$.  We proceed by induction on $i+j$.  If $i+j=0$, the result is true by Theorem \ref{thm:central_product}, and we have $H=B_A$, $K=B_B$, and $HK = B_AB_B = B_G$.  Suppose $i+j >0$ and suppose without loss of generality that $i > 0$.  Let $H_0 \in \CD(A)$ of height $i-1$ with $H_0 < H$.  Then by induction $H_0K \in \CD(G)$ is of height $i+j-1$.  Suppose $U \in \CD(G)$ with $H_0K < U \leq HK$ is arbitrary.  By Lemma \ref{lem: pi}, $U = \pi_A(U)K$.  By Proposition \ref{prop: almost}, $\pi_A(U)(A \cap B) = \pi_A(U) \in \CD(A)$.  And since $H_0, \pi_A(U),$ and $H$ all contain $A \cap B$, by Lemma \ref{lem: prod} we have that $H_0 < \pi_A(U) \leq H$.  Since $H_0 \in \CD(A)$ is of height $i-1$ and $H \in CD(A)$ is of height $i$, we have that $\pi_A(U) = H$, and so $U=HK$.  Since $U$ was arbitrary, it follows that $HK \in \CD(G)$ is of height $i+j$.

We now establish the assertion in regards to depth.  Let $n_A$, $n_B$, and $n_G$ be the heights of $\CD(A)$, $\CD(B)$, and $\CD(G)$, respectively.  Suppose $H\in \CD(A)$ has height $i$ and $K\in\CD(B)$ has height $j$.  Then we have just shown that $HK \in \CD(G)$ has height $i+j$.  And so $H\in \CD(A)$ has depth $n_A-i$ and $K\in\CD(B)$ has depth $n_B-j$, and $HK \in \CD(G)$ has depth $n_G-(i+j)$.  Now we have established that $n_A + n_B = n_G$, by applying the result to the heights of the top elements of $\CD(A)$, $\CD(B)$, and $\CD(G)$.  Observe that $n_A - i + n_B - j = n_A + n_B - (i+j) = n_G - (i+j)$, and we have the result in regards to depth.
\end{proof}

Our last application of this section relies on both Theorem \ref{thm:central_product} and Theorem \ref{thm: levels}.

\begin{proposition} \label{prop: heights_2_3}
Suppose that $G$ is a finite group with $G \in \CD(G)$ and suppose that $\CD(G)$ has height $2$ or $3$. Then $G$ admits no proper central product if and only if every atom in $\CD(G)$ is abelian.
\end{proposition}

\begin{proof}
One direction is true by Lemma \ref{lem: abel_atoms}.

Suppose that every atom in $\CD(G)$ is abelian, and let $G=AB$ be an arbitrary central product.  Since $G \in \CD(G)$, by Corollary \ref{cor: G_in_CD} we have that $A \in \CD(A)$ and $B \in \CD(B)$.

Now if one of $\CD(A)$ or $\CD(B)$ has height $0$, say $\CD(A)$, then $\CD(A) = \{A\}$ and so $A$ is abelian, and the result follows.

Otherwise, since $\CD(G)$ has height $2$ or $3$, we have that at least one of $\CD(A)$ or $\CD(B)$ has height $1$, say $\CD(A)$.  So $\CD(A) = \{A, \textbf{Z}(A) \}$.  By Theorem \ref{thm: levels} we have that $A\textbf{Z}(B)$ has height 1 in $\CD(G)$, and thus is a nonabelian atom in $\CD(G)$, a contradiction.
\end{proof}

The fun stops after height $3$.  Construct the central product $G = Q_8 * Q_8$, which is of course proper.  Note that $\CD(G)$ has height $4$, and all of the atoms have order $4$ so they are all abelian.  

\section{Investigating the Chermak--Delgado Lattice via the Central Product.}\label{sec: corollaries}

In this section we show how using the central product and the results from Section \ref{sec: main} we can obtain new proofs of some results about Chermak--Delgado lattice.  

Let $\mathcal{L}(G)$ denote the lattice of all subgroups of a group $G$.  Given a sublattice $\mathcal{C}$ of $\mathcal{L}(G)$, and given $H \leq K \leq G$, we denote the interval in $\mathcal{C}$ between $H$ and $K$ by $[H : K]_{\mathcal{C}} = \{ X \in \mathcal{C} \,\, | \,\, H \leq X \leq K\}$. For example, when $G = Q_8$ the quaternion group of order 8, then $[\langle -1 \rangle\ : G ]_{\mathcal{L}(G)}$ contains 5 groups: $\langle -1 \rangle ,\langle i \rangle, \langle j \rangle, \langle k\rangle,$ and $G$. 

The next lemma originates from \cite[Proposition 1.5]{Bre_12}.

\begin{lemma}\label{lem: interval_of_product}
Let $G$ be a finite group and let $H \in \CD(G)$.  Then $\CD(H\textbf{C}_G(H)) = [\textbf{Z}(H) : H\textbf{C}_G(H)]_{\CD(G)}$.
\end{lemma}

\begin{proof}
First note that $H\textbf{C}_G(H)$ is a central product and so $\textbf{Z}(H\textbf{C}_G(H)) = \textbf{Z}(H)\textbf{Z}(\textbf{C}_G(H))$.  

Now since $H \in \CD(G)$, we have that $\textbf{Z}(\textbf{C}_G(H)) = \textbf{C}_G(\textbf{C}_G(H)) \cap \textbf{C}_G(H) = H \cap \textbf{C}_G(H) = \textbf{Z}(H)$.  And so $\textbf{Z}(H\textbf{C}_G(H)) = \textbf{Z}(H)$.

We establish that for any $X$ with $\textbf{Z}(H) \leq X \leq H\textbf{C}_G(H)$, we have that $\textbf{C}_G(X) = \textbf{C}_{H\textbf{C}_G(H)}(X)$.

Suppose that $H \cap \textbf{C}_G(H) \leq X \leq H\textbf{C}_G(H)$.  Then $\textbf{C}_G(H\textbf{C}_G(H))  \leq \textbf{C}_G(X) \leq \textbf{C}_G(H \cap \textbf{C}_G(H))$, and since $H \in \CD(G)$, we have that $H \cap \textbf{C}_G(H) \leq \textbf{C}_G(X) \leq H\textbf{C}_G(H)$, and hence $\textbf{C}_G(X) = \textbf{C}_{H\textbf{C}_G(H)}(X)$.  

It follows that $X \in \CD(H\textbf{C}_G(H))$ if and only if $X \in [\textbf{Z}(H) : H\textbf{C}_G(H)]_{\CD(G)}$.
\end{proof}

The following corollary appears in An \cite[Theorem 3.4]{An_22} and \cite[Theorem 4.4]{An_22_2}. However, as stated at the start of the section, our goal is to show the power of the central product when investigating the Chermak--Delgado lattice. Of note, our proof utilizes Lemma \ref{lem: interval_of_product} and Proposition \ref{prop: almost} together with some standard results about the Chermak--Delgado lattice.

\begin{corollary}\label{cor: subgroups_in_CD}
Let $G$ be a finite group and let $H \in \CD(G)$.  Then $\CD(H) = [\textbf{Z}(H) : H]_{\CD(G)}$.
\end{corollary}

\begin{proof}
By Lemma \ref{lem: interval_of_product}, $\CD(H\textbf{C}_G(H)) = [\textbf{Z}(H) : H\textbf{C}_G(H)]_{\CD(G)}$.  

So for any $X \in [\textbf{Z}(H) : H]_{\CD(G)}$, $X \in \CD(H\textbf{C}_G(H))$.  Applying Proposition \ref{prop: almost} with the central product $H\textbf{C}_G(H)$, we have that $X(H \cap \textbf{C}_G(H)) = X\textbf{Z}(H) = X \in \CD(H)$.  Thus $[\textbf{Z}(H) : H]_{\CD(G)} \subseteq \CD(H)$. 

To see the reverse containment, note that every subgroup of $\CD(H)$ contains $\textbf{Z}(H)$, and every subgroup of $\CD(\textbf{C}_G(H))$ contains $\textbf{Z}(\textbf{C}_G(H)) = \textbf{C}_G(\textbf{C}_G(H)) \cap \textbf{C}_G(H) = H \cap \textbf{C}_G(H) = \textbf{Z}(H)$.  By Proposition \ref{prop: almost}, $\CD(H)\cdot \CD(\textbf{C}_G(H)) \subseteq \CD(H\textbf{C}_G(H))$, and since $\textbf{Z}(H) \in \CD(H\textbf{C}_G(H))$, it follows that $\textbf{Z}(H) \in \CD(H)$ and $\textbf{Z}(H) \in \CD(\textbf{C}_G(H))$.  And so $\CD(H) \subseteq \CD(H\textbf{C}_G(H))$.  And since $\CD(H\textbf{C}_G(H)) = [\textbf{Z}(H) : H\textbf{C}_G(H)]_{\CD(G)}$, we have that $\CD(H) \subseteq [\textbf{Z}(H) : H]_{\CD(G)}$.
\end{proof}

Of note, Corollary \ref{cor: subgroups_in_CD} can be used to argue that certain groups never appear in a Chermak--Delgado lattice for any group $G$.  For example, the alternating group $A_4$ can never be in a Chermak--Delgado lattice since $A_4 \notin \CD(A_4)$. 

Taking the central product approach also allows us to prove a result of Tărnăuceanu \cite[Corollary 4]{Tar_18}. Tărnăuceanu's result occurs as a corollary to their classification of groups $G$ satisfying $\CD(G) = [\textbf{Z}(G) : G]_{\mathcal{L}(G)}$. Once again, we are able to present a central product based proof. 

\begin{corollary}\label{cor: full_transitive}
If $G$ is a finite group satisfying $\CD(G) = [\textbf{Z}(G) : G]_{\mathcal{L}(G)}$ and $H$ is a subgroup of $G$, then $\CD(H) = [\textbf{Z}(H) : H]_{\mathcal{L}(G)}$.
\end{corollary}

\begin{proof}
By definition $\CD(H) \subseteq [\textbf{Z}(H) : H]_{\mathcal{L}(G)}$. For the reverse containment, suppose $X \in [\textbf{Z}(H) : H]_{\mathcal{L}(G)}$. 

Since $\textbf{Z}(G) \leq \textbf{C}_G(H)$, we have that $H\textbf{C}_G(H)\in \CD(G)$.  It follows from Corollary \ref{cor: subgroups_in_CD} that $\CD(H\textbf{C}_G(H)) = [\textbf{Z}(H\textbf{C}_G(H)) : H\textbf{C}_G(H)]_{\mathcal{L}(G)}=[\textbf{Z}(H)\textbf{Z}(\textbf{C}_G(H)) : H\textbf{C}_G(H)]_{\mathcal{L}(G)}$.

And so $X\textbf{C}_G(H) \in \CD(H\textbf{C}_G(H))$, and it follows by Proposition \ref{prop: almost} that $X(H \cap \textbf{C}_G(H)) = X\textbf{Z}(H) = X \in \CD(H)$.
\end{proof}


\bibliographystyle{plainnat}
\bibliography{ref}

\end{document}